\documentclass{amsart}
\usepackage[cp1251]{inputenc}
\usepackage[shorthands=off,english]{babel}
\usepackage{amsfonts}
\usepackage{tikz, tikz-cd}
\usepackage{amssymb,amsmath,amsthm,eucal,amscd}
\usepackage{mathtools}
\usepackage{enumitem}
\usepackage[figurename=Fig.]{caption}
\usetikzlibrary{patterns}
\usepackage{pgfplots}
\usepackage{relsize}
\usetikzlibrary{babel}

\usepackage{amsrefs}

\usepackage{mathtools}
\usepackage{url}
\usepackage{xcolor}
\usepackage{graphicx}

\usepackage{yhmath}
\mathtoolsset{showonlyrefs}
\newcommand{\Z}{\mathbb Z}

\newcommand{\R}{\mathbb R}
\newcommand{\C}{\mathbb C}

\newcommand{\T}{\mathbb T}
\newcommand{\quat}{\mathbb H}
\renewcommand{\P}{\mathbb P}

\newcommand{\mc}{\mathcal}

\newcommand{\lb}{\lbrace}
\newcommand{\rb}{\rbrace}
\newcommand{\la}{\langle}
\newcommand{\ra}{\rangle}
\renewcommand{\phi}{\varphi}

\DeclareMathOperator{\vertex}{vert}

\renewcommand{\leq}{\leqslant}
\renewcommand{\geq}{\geqslant}

\theoremstyle{plain}
\newtheorem{thm}{Theorem}[section]
\newtheorem{lm}[thm]{Lemma}
\newtheorem{cor}[thm]{Corollary}
\newtheorem{pr}[thm]{Proposition}

\theoremstyle{remark}
\newtheorem{rem}[thm]{Remark}
\newtheorem{ex}[thm]{Example}
\theoremstyle{definition}
\newtheorem{defn}[thm]{Definition}

\tikzset{
  every picture/.append style={
    execute at begin picture={\shorthandoff{"}},
    execute at end picture={\shorthandon{"}}
  }
}

\begin{document}
\date{}

\author{Grigory Solomadin}
\thanks{Moscow State University.
E-mail: \texttt{grigory.solomadin@gmail.com}.
The author is a Young Russian Mathematics award winner and would like to thank
its sponsors and jury.}
\title[Bordism of $H_{1,n}$ and $\C P^1\times \C P^{n-1}$]{The explicit geometric constructions of bordism of Milnor hypersurface $H_{1,n}$ and $\C P^1\times \C P^{n-1}$}

\begin{abstract}
In the present paper we construct two new explicit complex bordisms between any two projective bundles over $\C P^1$ of the same complex dimension, including the Milnor hypersurface $H_{1,n}$ and $\C P^1\times \C P^{n-1}$. These constructions reduce the bordism problem to the null-bordism of some projective bundle over $\C P^1$ with the non-standard stably complex structure.
\end{abstract}

\maketitle

\section{Introduction}

Let $P^{n}(a)$ denote the fiberwise projectivisation $\P(\eta_{1}^{a}\oplus\underline{\C}^{n-1})$ over $\C P^1$, where $\underline{\C}^{n-1}$ denotes the trivial complex linear bundle or rank $n-1$, $\eta_1\to\C P^1$ is the tautological bundle over $\C P^1$ and $\eta_{1}^{a}$ denotes the tensor product of $|a|$ copies of $\overline{\eta_{1}}$ for $a<0$ (the bar means complex conjugation) and of $\eta_{1}$ for $a\geq 0$, resp. $P^n(a)$ is a complex $n$-dimensional manifold. In the present paper we give three different proofs of the following well-known Theorem (see \cite[Theorem 3.1 b]{bu-mi-no-71}, \cite[Excercise D.6.12, p.471]{bu-pa-15}).

\begin{thm}\label{thm:main}
The manifolds $P^n(a)$ corresponding to different $a\in\Z$ are complex bordant to each other.
\end{thm}

Any $\C P^{n-1}$-bundle over $\C P^1$ is homeomorphic to $P^{n}(a)$ for some $a\in \Z$. Manifolds of such kind include the Milnor hypersurface $P^n(1)=H_{1,n}$ (i.e. a generic smooth hypersurface of bidegree $(1,1)$ in $\C P^1\times \C P^{n}$) and the trivial projective bundle $P^n(0)=\C P^1\times\C P^{n-1}$.

\begin{cor}
$H_{1,n}$ is complex bordant to $\C P^{n-1}\times\C P^1$.
\end{cor}

The first proof of Theorem \ref{thm:main} is an algebraic folklore result. Namely, one has to compute Chern numbers of $P^n(a)$ and notice that they do not depend on $a$ (see Proposition \ref{pr:hijnum}). On the other hand, Theorem \ref{thm:main} states that any two such complex manifolds are connected components of the boundary of some stably complex manifold (modulo its null-bordant boundary components). A question remaining open was to give an explicit geometric construction of a stably complex manifold with boundary components $P^n(a)$, $P^n(b)$ and null-bordant manifolds for any $a,b\in\Z$. This problem was discussed by M.~Kreck, P.~Landweber, T.~Panov and N.~Strickland since the beginning of $2000$ (private communication with T.~Panov, 2018).

The remaining two proofs of Theorem \ref{thm:main} are geometric, by providing an explicit bordism between the manifolds in question. These are new results in the field giving the solution to the aforementioned problem. Second proof of Theorem \ref{thm:main} is based on the refinement of the gluing lemma from \cite{to-00}. The original ``probably well known'' method from \cite{to-00} consisted of gluing three stably complex manifolds $A,B,C$ of dimension $2n$ with diffeomorphic boundaries to obtain the stably complex\\
$(2n+1)$-dimensional manifold with the respective boundary components $A\cup B$, $B\cup C$, $C\cup A$ with the natural stably complex structures. Our refinement of this method consists of taking the clutching functions into account. This results in obtaining different stably complex structures (depending on the clutching functions) on the glued manifold for the same stably complex manifolds $A,B,C$.

We apply the refined gluing lemma to three copies of $D^2\times\C P^{n-1}$ (where\\
$D^2\subset\R^{2}$ is a $2$-ball) with prescribed clutching functions on the boundary. For any two complex projective bundles $P^n(a)$, $P^{n}(b)$ over $\C P^1$ of complex dimension $n$ we indicate the clutching functions. The above construction for this data then yields the manifold $W^{2n+1}$ being a complex bordism (modulo the null-bordant boundary component) between $P^n(a)$, $P^{n}(b)$ due to Stong's classical ``abstract nonsense'' definition of bordism from \cite[p.5]{st-68} (see Theorem \ref{thm:gluenecc}).

The third proof of Theorem \ref{thm:main} is based on toric topology methods. We use Sarkar's construction \cite{sou-15} of a stably complex orbifold with quasitoric orbifold boundary to construct the necessary manifold $W^{2n+1}$. We choose the simple (combinatorial) polytope $Q$ to be the cartesian product $H^2\times\Delta^{n-1}$ of the hexagon and the symplex with three distinguished facets $F_2,F_4,F_6$. We show that for any two complex projective bundles over $\C P^1$ there exists an isotropy function $\lambda$ on the complement $\lb Q\setminus F_2,F_4,F_6\rb$ s.t. the corresponding orbifold is a stably complex manifold with boundary yielding the bordism between them (modulo the null-bordant boundary component), see Theorem \ref{thm:orbb}.

It is pleasure to acknowledge several enjoyable and influential discussions with  A.Ayzenberg, V.M.~Buchstaber, T.~Panov and S.~Sarkar. The author is a Young Russian Mathematics award winner and would like to thank its sponsors and jury.

\section{Topology of $\C P^k$-bundles over $\C P^1$}

In this Section we give a brief overview of some well-known topological properties of projective fibre bundles. We will use these properties in the next two Sections. We refer to  \cite{bu-pa-15} as a source on topology of toric and quasitoric manifolds.

\begin{pr}\label{pr:natp}
The natural complex structure on $\C P^n$ satisfies
\[
T\C P^n\oplus\underline{\C}\simeq (n+1)\overline{\eta_{n}},
\]
where $(n+1)\overline{\eta_{n}}$ denotes the Whitney sum of $n+1$ copies of the vector bundle $\overline{\eta_{n}}$.
\end{pr}

\begin{pr}
For any complex vector bundle $\xi\to \C P^1$ of rank $n$ the complex fiberwise projectivisation $\P(\xi)$ is homeomorphic to $P^{n}(a)$ for some $a\in\Z$.
\end{pr}

\begin{pr}\label{pr:projbndlstr}
The natural complex structure on the projectivisation $P^{n}(a)$ satisfies
\begin{equation}\label{eq:stand}
TP^{n}(a)\oplus\underline{\C}^2\simeq \overline{\zeta}\otimes (p^*\eta^a_{1}\oplus \underline{\C}^{n-1})\oplus 2p^*\overline{\eta_{1}},
\end{equation}
where $\zeta$ denotes the tautological line bundle on the projectivisation $P^{n}(a)$.
\end{pr}

In the following we will consider non-standard complex structures on $\C P^{1}$ and $P^{n}(a)$.

\begin{defn}
Let $c^{0}_{\mc{T}}:\ T \C P^{1}\oplus\underline{\R}^2\simeq_{\R} \underline{\C}^{2}$ and
\begin{equation}\label{eq:stcplxh}
c_{\mc{T}}(a):\ TP^{n}(a)\oplus\underline{\C}^2\simeq_{\R} \overline{\zeta}\otimes (p^*\eta^a_{1}\oplus \underline{\C}^{n-1})\oplus \underline{\C}^{2}.
\end{equation}
be the isomorphisms of realifications of vector bundles above.
\end{defn}

The stably complex manifold $(\C P^1,c^{0}_{\mc{T}})$ is easily seen to be null-bordant. Further we will check that $(P^{n}(a),c_{\mc{T}}(a))$ also bounds (see Proposition \ref{pr:hijnum} $(ii)$).

\begin{rem}
Consider the quaternionic space $\quat$ with coordinates $c_1,c_2\in\C$. The unit sphere $S^3\subset \quat$ is given by the condition $|c_1+c_2j|^2=1$. Left and right actions of $S^1$ on $\quat$ are given by the formulas
\[
L(e^{2\pi i t})(c_1+c_2j)=e^{2\pi i t}(c_1+c_2j)=e^{2\pi i t} c_1+e^{2\pi i t}c_2j,
\]
\[
R(e^{2\pi i t})(c_1+c_2j)=(c_1+c_2j)e^{2\pi i t}=e^{2\pi i t} c_1+e^{2\pi i t}\overline{c}_2j.
\]
$S^3\subset \quat$ is invariant under both of these actions. The quotients of $S^3$ by the left and right actions of $S^1$ are $\C P^1$ and $(\C P^1,c^{0}_{\mc{T}})$, resp.
\end{rem}

By applying the Leray-Hirsch theorem (see \cite[\S 20]{bo-tu-82}) to the particular projective fibre bundle $P^{n}(a)$ one immediately obtains
\begin{pr}\label{pr:coh}
One has an isomorphism
\[
H^*(P^{n}(a);\Z)\simeq\Z[x,y]/(x^2,y^{k+1}-axy^k),
\]
where $x=c_1(p^*\overline{\eta_1}), y=c_1(\overline{\zeta})$ in the denotations of Proposition \ref{pr:projbndlstr}.
\end{pr}

We give a simple calculation of Chern numbers $c_{I}$ of projective fibre bundles $P^{n}(a)$, $a\in\Z$, below.
\begin{pr}\label{pr:hijnum}
Let $I=\lb i_1,\dots, i_r\rb$ be a partition of $n=\sum_{l=1}^{r} i_l$, and let $a,b\in\Z$.

$(i)$ One has the identity
\[
c_{I}(P^{n}(a))=2\prod_{s=1}^r \binom{n}{i_s}\sum_{q=1}^{r}\frac{i_q}{n+1-i_q};
\]

$(ii)$ One has the identity $c_I(P^{n}(a),c_{\mc{T}}(a))=0$. The stably complex manifold $(P^{n}(a),c_{\mc{T}}(a))$ bounds;

$(iii)$ One has the identity
\[
[P^{n}(a)]-[P^{n}(b)]+[P^{n}(b-a),c_{\mc{T}}(b-a)]=0
\]
in the unitary bordism ring $\Omega_{*}^{U}$.
\end{pr}
\begin{proof}
$(i)$. The Chern class of $TP^n(a)$ is equal to (see \eqref{eq:stand})
\begin{multline}
\sum_{q=0}^{n}c_q=(1+x)^2(1-ax+y)(1+y)^{n-1}=\\
=1+\sum_{q=0}^{n}\biggl(\binom{n}{q}y^q+(2-a)\binom{n-1}{q-1}xy^{q-1}+2\binom{n-1}{q-2}xy^{q-1}\biggr).
\end{multline}
Then by Proposition \ref{pr:coh} one has
\begin{multline}
c_{I}(P^{n}(a))=\la c_{i_1}\cdots c_{i_r}, [P^n(a)]\ra=\\
=a\prod_{l=1}^r\binom{n}{i_l}+\sum_{q=1}^r\biggl((2-a)\binom{n-1}{i_q-1}+2\binom{n-1}{i_q-2}\biggr)\prod_{s\neq q}\binom{n}{i_s}=\\
=2\sum_{q=1}^r\binom{n}{i_q-1}\prod_{s\neq q}\binom{n}{i_s}=2\prod_{s=1}^r \binom{n}{i_s}\cdot\sum_{q=1}^{r}\frac{i_q}{n+1-i_q},
\end{multline}
where by abuse of the notation $\la *, [P^{n}(a)]\ra$ denotes the canonical pairing with the fundamental class\\ $[P^{n}(a)]\in H_{2(k+1)}(P^{n}(a);\Z)$.

$(ii)$. The Chern class of the stably complex structure on $(P^{n}(a),c_{\mc{T}}(a))$ is equal to (see \eqref{eq:stcplxh}):
\[
\sum_{q=0}^{n}c_q=(1+y-ax)(1+y)^{n-1}=1+\sum_{q=1}^{n} \biggl(\binom{n}{q}y^q-a\binom{n-1}{q-1}xy^{q-1}\biggr).
\]
Then one has
\begin{multline}
c_{I}(P^{n}(a),c_{\mc{T}}(a))=\la c_{i_1}\cdots c_{i_r}, [P^{n}(a)]\ra=
a\prod_{s=1}^r\binom{n}{i_s}-a\sum_{q=1}^r\binom{n-1}{i_q-1}\prod_{s\neq q}\binom{n}{i_s}=\\
=a\prod_{s=1}^r\binom{n}{i_s}-a\prod_{s= 1}^r\binom{n}{i_s}\sum_{q=1}^r\frac{i_q}{n}=0.
\end{multline}
The claim about complex bordism class of $(P^{n}(a),c_{\mc{T}}(a))$ follows from theorem of Milnor and Novikov about the complex bordism ring, see \cite[p.117]{st-68}.

$(iii)$. Follows from $(i)$ and $(ii)$ immediately.
\end{proof}

\begin{pr}[See {\cite[Section 7.8]{bu-pa-15}}, {\cite[Example 2.4]{lu-pa-14}}]\label{pr:quas}
$P^{n}(a)$ and $(P^{n}(a),c_{\mc{T}}(a))$ are toric and quasitoric manifolds, resp., having the moment polytopes combinatorially equivalent to $\Delta^1\times\Delta^{n-1}$, where $\Delta^{n-1}$ denotes the $(n-1)$-dimensional simplex in $\R^{n-1}$. For any $b\in\Z$ the characteristic matrices of $P^{n}(a)$ and $(P^{n}(a),c_{\mc{T}}(a))$ are $GL_{n}(\Z)$-equivalent to
\[
\begin{pmatrix}
1   & -1 &  0 &  0 &\dots&  0 &  0 \\
b   &  b-a &  1 &  0 &\dots&  0 & -1 \\
0   &  0 &  0 &  1 &\dots&  0 & -1 \\
\vdots&\vdots&\vdots&\vdots&\ddots&\vdots&\vdots\\
0   &  0 &  0 &  0 &\dots&  1 & -1 \\
\end{pmatrix},
\begin{pmatrix}
-1  & -1 &  0 &  0 &\dots&  0 &  0 \\
b   &  b-a &  1 &  0 &\dots&  0 & -1 \\
0   &  0 &  0 &  1 &\dots&  0 & -1 \\
\vdots&\vdots&\vdots&\vdots&\ddots&\vdots&\vdots\\
0   &  0 &  0 &  0 &\dots&  1 & -1 \\
\end{pmatrix},
\]
respectively, in the following facet ordering of of $\Delta^1\times\Delta^{n-1}$: $\Delta^{0}_{1}\times\Delta^{n-1}$, $\Delta^{0}_{2}\times\Delta^{n-1}$, $\Delta^1\times\Delta^{n-2}_{1},\dots,\Delta^1\times\Delta^{n-2}_{n}$ where $\Delta^{n-2}_{1},\dots,\Delta^{n-2}_{n}$ denote the facets of $\Delta^{n-1}$.
\end{pr}

We remind the clutching construction for vector bundles (\cite[pp.20--23]{at-67}).

\begin{defn}
Let $X_1,X_2$ be compact topological spaces with complex vector bundles $\xi_1\to X_1,\ \xi_2\to X_2$. Suppose that there is given an isomorphism\\
$\phi:\ (\xi_1)|_{Y}\to (\xi_2)|_{Y}$ of vector bundles over the intersection $Y=X_1\cap X_2$. Then the (locally trivial) vector bundle $\xi_1\cup_{\phi}\xi_2\to X=X_1\cup X_2$ is defined as the quotient of $\xi_1\oplus \xi_2$ by the equivalence relation identifying the vector $e\in(\xi_1)|_{Y}$ with the vector $\phi(e)\in(\xi_2)|_{Y}$.
\end{defn}

The clutching construction for real vector bundles is defined in a similar way. 

\begin{ex}\label{ex:s2}
Take two $2$-dimensional balls $D_1,\ D_2$ with trivial linear vector bundles $\underline{\C}\to D_1,\ \underline{\C}\to D_2$. Embed the positively oriented circle\\
$S^1=\lb e^{2\pi i t}|\ t\in[0,1)\rb\subset\C$. Consider the isomorphism $f^a: \underline{\C}\to\underline{\C}$ of vector bundles over $S^1$ given by the formula $f(z,w)=(z,z^a\cdot w)$, $a\in\Z$. Clearly, $f^a\circ f^b=f^{a+b}$ for any $a,b\in\Z$. The union along the boundaries $D_1\cup_{S^1} D_2$ is the $2$-dimensional sphere $S^2$. The vector bundle $\underline{\C}\cup_{f^a}\underline{\C}\to S^2$ is isomorphic to the pull-back of $\overline{\eta_1^a}\to\C P^1$ under the homeomorphism $S^2\simeq \C P^1$ (see \cite[p.49]{at-67}). Hence, by Proposition \ref{pr:natp}, the vector bundle $\underline{\C}^2\cup_{f^2\oplus Id}\underline{\C}^2\to S^2$ is isomorphic to the pull-back of $\overline{\eta_1^2}\oplus\underline{\C}\simeq 2\overline{\eta_1}\simeq T\C P^1\oplus\underline{\C}\to\C P^1$ under the homeomorphism $S^2\simeq \C P^1$
\end{ex}

\begin{ex}\label{ex:bundle}
Consider the projectivisation $p:\ P^n(a)\to\C P^1$. The restriction of $\overline{\zeta}\otimes (p^*\eta^a_{1}\oplus \underline{\C}^{n-1})\to P^n(a)$ to the preimage $p^{-1}(S^1)=S^1\times \C P^{n-1}$ of the equatorial circle coincides with the restriction of $\overline{\zeta}\otimes \underline{\C}^{n}$. Denote by $F_{n}(a)$ the isomorphism $f^a\oplus Id:\ \underline{\C}^{n}\to \underline{\C}^{n}$ of vector bundles, i.e. the clutching function of the vector bundle $\overline{\eta_1^a}\oplus\underline{\C}^{n-1}\to\C P^1\simeq S^2$ w.r.t. the union $D_1\cup_{S^1} D_2$. Take two copies $X_1,X_2$ of the cartesian product $D^2\times\C P^{n-1}$, so that $X_1\cap X_2=S^1\times\C P^{n-1}$.  Then the clutching function of $\overline{\zeta}\otimes (p^*\eta^a_{1}\oplus \underline{\C}^{n-1})\oplus 2p^*\overline{\eta}_{1}\to P^n(a)$ appearing in the natural complex structure on $P^n(a)$ (see \eqref{eq:stand}) w.r.t. the decomposition $X_1\cup X_2$ is the isomorphism
\[
Id\otimes F_{n}(-a)\oplus F_{2}(2):\ \overline{\zeta}\otimes \underline{\C}^{n}\oplus\underline{\C}^2\to \overline{\zeta}\otimes \underline{\C}^{n}\oplus\underline{\C}^2
\]
of vector bundles over $S^1\times\C P^{n-1}$. Next, the clutching function of\\
$\overline{\zeta}\otimes (p^*\eta_{1}^{a}\oplus \underline{\C}^{n-1})\oplus \underline{\C}^2$ appearing in the non-standard stably complex structure $c_{\mc{T}}(a)$ on $P^{n}(a)$ (see \eqref{eq:stcplxh}) w.r.t. the decomposition $X_1\cup X_2$ is the isomorphism
\[
Id\otimes F_{n}(-a)\oplus F_{2}(0):\ \overline{\zeta}\otimes \underline{\C}^{n}\oplus\underline{\C}^2\to \overline{\zeta}\otimes \underline{\C}^{n}\oplus\underline{\C}^2
\]
of vector bundles over $S^1\times\C P^{n-1}$.
\end{ex}

\section{Geometric construction of bordism using the gluing lemma}

In this Section we use the gluing lemma (Lemma \ref{lm:glue} below) in unitary bordism theory to produce the explicit bordism between any two projective fibre bundles over $\C P^1$ of the same dimension. 

First remind the construction of the opposite stably complex structure on a stably complex manifold.
\begin{defn}
Let $(M^{2n},c_{\mc{T}})$ be a stably complex compact manifold (closed or with boundary) of dimension $2n$ with the real vector bundle isomorphism
\[
c_{\mc{T}}:\ TM\oplus\underline{\R}^{2(m-n)}\to \xi,
\]
where $\xi$ is a complex vector bundle over $M$. Then the stably complex manifold $(M^{2n},\overline{c_{\mc{T}}})$ is defined as the smooth manifold $M$ with opposite orientation, endowed with another stably complex structure $\overline{c_{\mc{T}}}$ given by the composition
\[
\overline{c_{\mc{T}}}:\ TM\oplus\underline{\R}^{2(m-n+1)}\overset{c_{\mc{T}}\oplus Id}{\longrightarrow} \xi\oplus\underline{\C}\overset{Id\oplus J}{\longrightarrow} \xi\oplus\underline{\C},
\]
where $J:\ \underline{\C}\to \underline{\C}$ is the fiberwise operator of complex conjugation over $M$. We denote $(M^{2n},\overline{c_{\mc{T}}})$ by $\overline{M}$.
\end{defn}
We remark that $[\overline{M}]=-[M]$ in the complex bordism ring.

Let $(X_i,c_{\mc{T},i},\xi_i)$ be a stably complex compact manifold of dimension $2n$ with boundary diffeomorphic to a smooth manifold $M$, so that
\[
c_{\mc{T},i}:\ TX_i\oplus\underline{\R}^{2(m-n)}\simeq \xi_i,
\]
where $i=1,2$. (W.l.g. we assume that complex vector bundles $\xi_1,\xi_2$ have ranks equal to $m$ by adding trivial vector bundles.) The collar lemma implies that the normal bundles of the natural inclusions of $M$ to $X_1, X_2$ are trivial. Hence,\\
$X_1\cup_{M} X_2$ is a smooth $2n$-dimensional manifold. In order to induce a stably complex structure on $X_1\cup_{M} X_2$ from $X_1, X_2$ we introduce the notion of a $\tau$-isomorphism.

Let $f:\ (\xi_1)|_{M}\to (\xi_2)|_{M}$ be an isomorphism of complex vector bundles. One has $(T X_1)|_M=(T X_2)|_M=TM$, so we define $g=(c_{\mc{T},2})|_{M}^{-1}\circ f\circ(c_{\mc{T},1})|_{M}$ completing the following commutative diagram
\begin{equation}\label{eq:comm0}
\begin{tikzcd}
TM\oplus\underline{\R}^{2(m-n)}\arrow{r}{(c_{\mc{T},1})|_{M}}\arrow{d}{g} & (\xi_{1})|_{M}\arrow{d}{f}\\
TM\oplus\underline{\R}^{2(m-n)}\arrow{r}{(c_{\mc{T},2})|_{M}} & (\xi_{2})|_{M}
\end{tikzcd}
\end{equation}
of isomorphisms of \textit{real} vector bundles. Due to commutativity of \eqref{eq:comm0} and \cite[p.22]{at-67} we have the isomorphism
\[
\phi:\ (T X_1\oplus\underline{\R}^{2(m-n)})\cup_{g} (TX_2\oplus\underline{\R}^{2(m-n)})\simeq \xi_1\cup_{f}\xi_2,
\]
of \textit{real} vector bundles defined as $\iota_1$ over $X_1$ and $\iota_2$ over $X_2$.

\begin{defn}
Any isomorphism
\[
\tau:\ T(X_1\cup_{M} X_2)\oplus\underline{\R}^{2(m-n)}\simeq (T X_1\oplus\underline{\R}^{2(m-n)})\cup_{g} (TX_2\oplus\underline{\R}^{2(m-n)}),
\]
of real vector bundles is called the $\tau$-isomorphism for the data\\
$(X_1,\iota_1,\xi_1)$, $(X_2,\iota_2,\xi_2)$, $f$.
\end{defn}

\begin{pr}\label{pr:tau}
Suppose that there exists a $\tau$-isomorphism $\tau$ for the data\\
$(X_1,\iota_1,\xi_1)$, $(X_2,\iota_2,\xi_2), f$. Then the isomorphism
\[
\phi\circ\tau:\ T(X_1\cup_{M} X_2)\oplus\underline{\R}^{2(m-n)}\to \xi_1\cup_{f}\xi_2
\]
of real vector bundles defines the stably complex structure on the manifold $X_1\cup_{M} X_2$.
\end{pr}

If there is a stably complex structure on $X_1\cup_{M} X_2$ restricting to the given stably complex structures on $X_1$, $X_2$ with the given clutching function $f$, then the $\tau$-isomorphism for the corresponding data exists.

Now we give the refined version of the gluing lemma from \cite{to-00} in terms of clutching functions of vector bundles. Let $(A,c_{\mc{T},A},\xi_A),(B,c_{\mc{T},B},\xi_B),(C,c_{\mc{T},C},\xi_C)$ be compact stably complex $2n$-dimensional manifolds with boundary diffeomorphic to a stably complex manifold $(M,c_{\mc{T},M},\xi_M)$. (W.l.g. we assume that $\xi_A,\xi_B,\xi_C,\xi_M$ have ranks equal to $m$ by adding trivial vector bundles.) For instance, one has
\[
c_{\mc{T},A}:\ TA\oplus\underline{\R}^{2(m-n)}\simeq_{\R} \xi_A.
\]
Let $f_A:\ (\xi_A)|_{\partial A}\to\xi_M,\ f_B:\ (\xi_B)|_{\partial B}\to\xi_M$ and $f_C:\ (\xi_C)|_{\partial C}\to\xi_M$ be isomorphisms of complex vector bundles. Let
\[
f_{A,B}:=f_B^{-1}\circ f_A,\ f_{B,C}:=f_C^{-1}\circ f_B,\ f_{C,A}:=f_A^{-1}\circ f_C.
\]
We define
\[
g_{A,B}=(c_{\mc{T},B})|_{M}^{-1}\circ f_{A,B}\circ(c_{\mc{T},A})|_{M},\ g_{B,C}=(c_{\mc{T},C})|_{M}^{-1}\circ f_{B,C}\circ(c_{\mc{T},B})|_{M},
\]
\[
g_{C,A}=(c_{\mc{T},A})|_{M}^{-1}\circ f_{C,A}\circ(c_{\mc{T},C})|_{M}.
\]
The isomorphisms
\[
\phi_{A,B}:\ (T A\oplus\underline{\R}^{2(m-n)})\cup_{g_{A,B}} (TB\oplus\underline{\R}^{2(m-n)})\simeq \xi_A\cup_{f_{A,B}}\xi_B,
\]
\[
\phi_{B,C}:\ (T B\oplus\underline{\R}^{2(m-n)})\cup_{g_{B,C}} (TC\oplus\underline{\R}^{2(m-n)})\simeq \xi_B\cup_{f_{B,C}}\xi_C,
\]
\[
\phi_{C,A}:\ (T C\oplus\underline{\R}^{2(m-n)})\cup_{g_{C,A}} (TA\oplus\underline{\R}^{2(m-n)})\simeq \xi_C\cup_{f_{C,A}}\xi_A,
\]
exist and are well-defined.

Now we would like to have the following $\tau$-isomorphisms:
\[
\tau_{A,B}:\ T(A\cup B)\oplus\underline{\R}^{2(m-n)}\simeq (T A\oplus\underline{\R}^{2(m-n)})\cup_{g_{A,B}} (TB\oplus\underline{\R}^{2(m-n)}),
\]
\[
\tau_{B,C}:\ T(B\cup C)\oplus\underline{\R}^{2(m-n)}\simeq (T B\oplus\underline{\R}^{2(m-n)})\cup_{g_{B,C}} (TC\oplus\underline{\R}^{2(m-n)}),
\]
\[
\tau_{C,A}:\ T(C\cup A)\oplus\underline{\R}^{2(m-n)}\simeq (T C\oplus\underline{\R}^{2(m-n)})\cup_{g_{C,A}} (TA\oplus\underline{\R}^{2(m-n)}),
\]
in order to induce the stably complex structures on smooth manifolds $A\cup_{M} B$, $B\cup_{M} C$, $C\cup_{M} A$. 

\begin{lm}\label{lm:glue}
Suppose that the $\tau$-isomorphisms $\tau_{A,B},\tau_{B,C},\tau_{C,A}$ exist. Then the stably complex manifold $W^{2n+1}$ obtained from the product $M\times H$ of $M$ and the $6$-gon $H\subset\R^2$ by attaching $A\times[0,1], B\times[0,1]$ and $C\times[0,1]$ to $M\times H_1, M\times H_3$ and $M\times H_5$, resp., is well-defined, where $H_1,\dots,H_6$ are edges of $H$ in the counterclockwise order. The connected boundary components of $W$ are stably complex manifolds $A\cup_{M} B$, $B\cup_{M} C$, $C\cup_{M} A$ with stably complex structures $c_{\mc{T},A,B}=\phi_{A,B}\circ\tau_{A,B}$, $c_{\mc{T},B,C}=\phi_{B,C}\circ\tau_{B,C}$, $c_{\mc{T},C,A}=\phi_{C,A}\circ\tau_{C,A}$, resp. In particular, one has the relation
\begin{equation}\label{eq:totaro}
[A\cup_{M} B, c_{\mc{T},A,B}]+[B\cup_{M} C, c_{\mc{T},B,C}]+[C\cup_{M} A, c_{\mc{T},C,A}]=0
\end{equation}
in the bordism group $\Omega^{U}_{2n}$.
\end{lm}
\begin{proof}
The construction of the manifold $W$ is given in \cite[Lemma 2.1]{to-00}. The other claims follow from Proposition \ref{pr:tau}.
\end{proof}

We stress that the stably complex structures on $A\cup_{M} B$, $B\cup_{M} C$, $C\cup_{M} A$ from Lemma \ref{lm:glue} depend not only on the stably complex structures on $A,B,C$ but also on the vector bundle isomorphisms $f_A, f_B, f_C$.

\begin{ex}
Take three copies of $2$-ball $D^2$ as $A,B,C$, resp. Choose\\
$f_A=f^1$, $f_B=f^{-1}$, $f_C=f^{-1}$. The manifolds $A\cup_{M} B$, $B\cup_{M} C$, $C\cup_{M}A$ are spheres $S^2$ with different stably complex structures corresponding to the vector bundles $\overline{\eta_1^2}\oplus\underline{\C}$, $\underline{\C}^2$, $\eta_1^2\oplus\underline{\C}$. Their respective clutching functions are $f^2\oplus Id, f^0\oplus Id, f^{-2}\oplus Id$ (see Example \ref{ex:s2}). Clearly, the corresponding $\tau$-isomorphisms exist. The relation \eqref{eq:totaro} takes form
\begin{equation}
[\C P^1]+[\C P^1, c^{0}_{\mc{T}}]-[\C P^1]=0,
\end{equation}
which indeed holds in $\Omega^{U}_{2}$.
\end{ex}

\begin{thm}\label{thm:gluenecc}
Let $a,b\in\Z$. Then Lemma \ref{lm:glue} yields the stably complex manifold $W^{2n+1}$ with boundary consisting of disjoint union of stably complex manifolds $P^{n}(a),\overline{P^{n}(b)}$ and $(P^{n}(b-a),c_{\mc{T}}(b-a))$ for some data.
\end{thm}
\begin{proof}
Substitute three copies of $D^2\times\C P^{n-1}$ into $A,B,C$, so that\\
$M=S^1\times\C P^{n-1}$. Substitute $Id\otimes F_n(0)\oplus F_2(1)$, $Id\otimes F_n(a)\oplus F_2(-1)$,\\
$Id\otimes F_n(b)\oplus F_2(-1)$ into $f_A,f_B,f_C$, resp. It follows from Example \ref{ex:bundle} that the corresponding $\tau$-isomorphisms exist. Hence, Lemma \ref{lm:glue} applies and yields the stably complex manifold from the claim.
\end{proof}

\section{Geometric construction of bordism using manifolds with quasitoric boundary}

In this Section we utilize the Sarkar's construction of orbifolds with quasitoric orbifold boundary \cite{sou-15} in order to give another explicit complex bordism between any two projective fibre bundles over $\C P^1$ of the same dimension.

We follow the notions and denotations from \cite{sou-15}.

\begin{defn}
An $(n+1)$-dimensional (combinatorial) simple polytope $Q\subset\R^{n+1}$ is said to be a polytope with exceptional facets $Q_1,\dots,Q_k$, if $Q_i\cap Q_j$ is empty for any $1\leq i< j\leq k$, and $\vertex Q=\bigcup \vertex Q_i$, where $\vertex Q$ denotes the set of vertices of $Q$. A simple polytope $Q$ with exceptional facets $Q_1,\dots,Q_k$ is denoted by $\lb Q\setminus Q_1,\dots, Q_k\rb$.
\end{defn}

Let $\mc{F}(Q)=\lb F_1,\dots, F_m\rb$ be the set of all facets of the polytope $Q$.

\begin{defn}
A function $\lambda:\mc{F}(Q)\to\Z^n$ is called an isotropy function on $\lb Q\setminus Q_1,\dots, Q_k\rb$, if the vectors $\lambda(F_{i_1}),\dots,\lambda(F_{i_r})$ are linearly independent in $\Z^n$ whenever the intersection of the facets $F_{i_1},\dots, F_{i_r}$ is nonempty. The vector $\lambda(F_i)$ is called an isotropy vector assigned to the facet $F_i$, $i=1,\dots,m$.
\end{defn}

Let $F$ be a non-empty codimension $l$ face in $Q$. If $F=P$, then we let $M(P)=\Z^n$. Otherwise, $0<l\leq n+1$. If $F$ is a face of $Q_i$ for some $i\in\lb 1,\dots,k\rb$, then $F$ is the intersection of a unique collection of $l$ facets $F_{i_1},\dots, F_{i_{l-1}},Q_i$ of $Q$. Otherwise, $F$ is the intersection of $l$ facets $F_{i_1},\dots, F_{i_{l}}$ of $Q$. Let
\[
M(F)=
\begin{cases}
\Z\la \lambda(F_{i_j})|\ j=1,\dots,l-1\ra,\ if\ F=F_{i_1}\cap\dots\cap F_{i_{l-1}}\cap Q_i,\\
\Z\la \lambda(F_{i_j})|\ j=1,\dots,l\ra,\ if\ F=F_{i_1}\cap\dots\cap F_{i_{l}},
\end{cases}
\]
having $M(F)\subseteq \Z^n$, where $\Z\la*\ra$ denotes the $\Z$-linear hull of a set in $\Z^n$.

For any non-empty face $F$ of codimension $l$ in $Q$ consider a compact torus
\[
\T^n_{M(F)}=(M(F)\otimes\R)/M(F)
\]
of dimension $l-1$ or $l$ depending on the situation of the face $F$. The inclusion $M(F)\subseteq \Z^n$ induces the natural homomorphism
\[
f_F:\ \T_{M(F)}\subseteq \T^n
\]
for any face $F\subset Q$, where $\T^n=(\Z^n\otimes\R)/\Z^n$. Denote the image of $f_F$ by $Im(f_F)$.
\begin{defn}
$W(Q,\lambda)$ is the quotient space
\[
Q\times \T^{n}/\sim,
\]
w.r.t. the equivalence relation
\[
(x,t_1)\sim (x,t_2)\Leftrightarrow t_1^{-1}t_2\in Im(f_F),
\]
where $F$ is the unique face of $Q$ containing $x$ in its relative interior. The space $W(Q,\lambda)$ is a $\T$-space where the action is induced by left action of $\T^n$ on itself. Let
\[
\pi: W(Q,\lambda)\to Q
\]
be the projection map defined by the formula $\pi([x,t])=x$. We consider the standard orientation of $\T^n$ and orientation on $Q$ induced from the ambient space $\R^{n+1}$.
\end{defn}

Consider an $(n+1)$-dimensional simple polytope $\lb Q\setminus Q_1,\dots, Q_k\rb$ in $\R^{n+1}$ with distinguished facets, and an isotropy function $\lambda:\mc{F}\to\Z^n$. The space $W(Q,\lambda)$ is a $(2n+1)$-dimensional orbifold with boundary consisting of qusitoric orbifolds, generally speaking. Define a function $\xi_i:\mc{F}(Q_i)\to\Z^n$ by the formula
\begin{equation}\label{eq:charf}
\xi_i(Q_i\cap F_j)=\lambda(F_j).
\end{equation}
The function $\xi_i$ is easily seen to be a di-characteristic function on $Q_i$ (see \cite{sou-15}).

\begin{thm}[{\cite[Corollary 4.5, Theorem 5.5]{sou-15}}]\label{thm:sarkar}
Suppose that $\lambda(F_{i_1}),\dots,\lambda(F_{i_r})$ is a part of a basis in $\Z^n$ whenever the intersection $F_{i_1}\cap\dots\cap F_{i_r}$ is nonempty. Then $W(Q,\Lambda)$ is a smooth stably complex $(2n+1)$-dimensional manifold with the boundary $\bigsqcup M(Q_i,\xi_i)$. The function $\xi_i$ is a characteristic function on $Q_i$, so $M(Q_i,\xi_i)$ is a quasitoric manifold, $i=1,\dots,k$. The restriction of the stably complex structure on $W(Q,\Lambda)$ to $M(Q_i,\xi_i)$ coincides with the canonical stably complex structure on $M(Q_i,\xi_i)$. 
\end{thm}

One can avoid the direct check of the isotropy property of the function $\lambda$ due to the following simple

\begin{lm}\label{lm:check}
Let $\lb Q\setminus Q_1,\dots,Q_k\rb$ be a simple $(n+1)$-dimesional polytope with a function $\lambda:\mc{F}(Q)\setminus\lb Q_1,\dots,Q_k\rb\to\Z^n$. Let $\xi_i:\mc{F}(Q_i)\to\Z^n$ be the respective restrictions of $\lambda$ to the facets of $Q_i$ given by \eqref{eq:charf}, $i=1,\dots,k$. Suppose that $\xi_i$ is a di-characteristic (characteristic, resp.) function for any $i=1,\dots,k$. Then $\lambda$ is an isotropy function (satisfying the condition of  Theorem \ref{thm:sarkar}, resp.).
\end{lm}

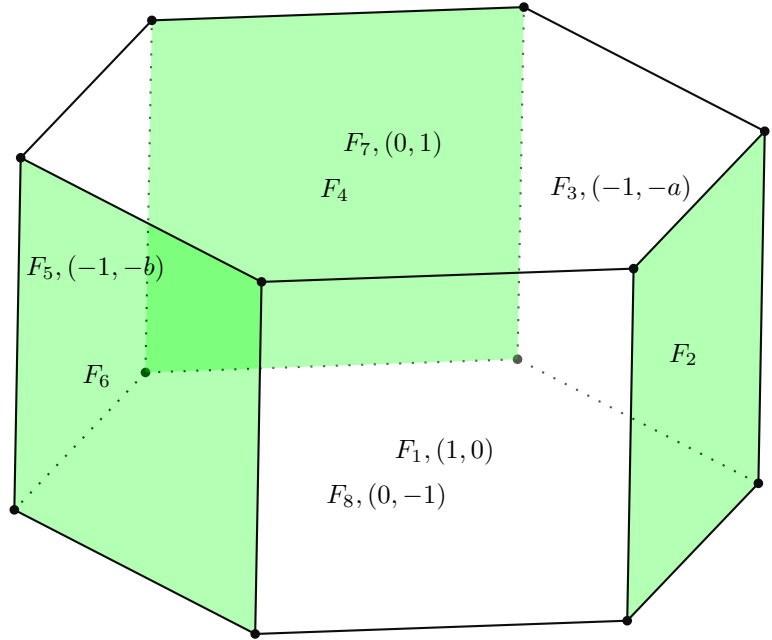
\begin{figure}
\centering
\caption{Polytope $Q$ with distinguished facets (green) and isotropy function for $n=2$.\label{fig:pic}}
\begin{tikzpicture}%
	[x={(-0.145821cm, 0.347936cm)},
	y={(0.989167cm, 0.035285cm)},
	z={(-0.016900cm, -0.936854cm)},
	scale=10.000000,
	back/.style={loosely dotted, thick},
	edge/.style={color=white!5!black, thick},
	facet/.style={fill=white, fill opacity=0.300000},
	facetg/.style={fill=green, fill opacity=0.300000},
	facetb/.style={fill=blue, fill opacity=0.300000},
	facetlast/.style={fill=white!50!black,fill opacity=0.300000},
	vertex/.style={inner sep=1pt,circle,draw=white!5!black,fill=white!5!black,thick,anchor=base},
    vertexr/.style={inner sep=1pt,circle,draw=red!5!black,fill=white!5!black,thick,anchor=base}]	

%
%
\coordinate (-0.50000, 0.25000, -0.25000) at (-0.50000, 0.25000, -0.25000);
\coordinate (-0.50000, 0.25000, 0.25000) at (-0.50000, 0.25000, 0.25000);
\coordinate (0.00000, 0.50000, -0.25000) at (0.00000, 0.50000, -0.25000);
\coordinate (0.00000, 0.50000, 0.25000) at (0.00000, 0.50000, 0.25000);
\coordinate (0.50000, -0.25000, 0.25000) at (0.50000, -0.25000, 0.25000);
\coordinate (0.50000, -0.25000, -0.25000) at (0.50000, -0.25000, -0.25000);
\coordinate (0.50000, 0.25000, 0.25000) at (0.50000, 0.25000, 0.25000);
\coordinate (0.50000, 0.25000, -0.25000) at (0.50000, 0.25000, -0.25000);
\coordinate (0.00000, -0.50000, -0.25000) at (0.00000, -0.50000, -0.25000);
\coordinate (0.00000, -0.50000, 0.25000) at (0.00000, -0.50000, 0.25000);
\coordinate (-0.50000, -0.25000, -0.25000) at (-0.50000, -0.25000, -0.25000);
\coordinate (-0.50000, -0.25000, 0.25000) at (-0.50000, -0.25000, 0.25000);
\draw[edge,back] (0.00000, 0.50000, 0.25000) -- (0.50000, 0.25000, 0.25000);
\draw[edge,back] (0.50000, -0.25000, 0.25000) -- (0.50000, -0.25000, -0.25000);
\draw[edge,back] (0.50000, -0.25000, 0.25000) -- (0.50000, 0.25000, 0.25000);
\draw[edge,back] (0.50000, -0.25000, 0.25000) -- (0.00000, -0.50000, 0.25000);
\draw[edge,back] (0.50000, 0.25000, 0.25000) -- (0.50000, 0.25000, -0.25000);
\node[vertex] at (0.50000, -0.25000, 0.25000)     {};
\node[vertex] at (0.50000, 0.25000, 0.25000)     {};
\fill[facet] (-0.50000, -0.25000, 0.25000) -- (-0.50000, 0.25000, 0.25000) -- (-0.50000, 0.25000, -0.25000) -- (-0.50000, -0.25000, -0.25000) -- cycle {};
\fill[facetg] (-0.50000, -0.25000, 0.25000) -- (0.00000, -0.50000, 0.25000) -- (0.00000, -0.50000, -0.25000) -- (-0.50000, -0.25000, -0.25000) -- cycle {};
\fill[facetg] (0.00000, 0.50000, 0.25000) -- (-0.50000, 0.25000, 0.25000) -- (-0.50000, 0.25000, -0.25000) -- (0.00000, 0.50000, -0.25000) -- cycle {};
\fill[facet] (-0.50000, -0.25000, -0.25000) -- (-0.50000, 0.25000, -0.25000) -- (0.00000, 0.50000, -0.25000) -- (0.50000, 0.25000, -0.25000) -- (0.50000, -0.25000, -0.25000) -- (0.00000, -0.50000, -0.25000) -- cycle {};
\fill[facetg] (0.50000, -0.25000, 0.25000) -- (0.50000, 0.25000, 0.25000) -- (0.50000, 0.25000, -0.25000) -- (0.50000, -0.25000, -0.25000) -- cycle {};
\draw[edge] (-0.50000, 0.25000, -0.25000) -- (-0.50000, 0.25000, 0.25000);
\draw[edge] (-0.50000, 0.25000, -0.25000) -- (0.00000, 0.50000, -0.25000);
\draw[edge] (-0.50000, 0.25000, -0.25000) -- (-0.50000, -0.25000, -0.25000);
\draw[edge] (-0.50000, 0.25000, 0.25000) -- (0.00000, 0.50000, 0.25000);
\draw[edge] (-0.50000, 0.25000, 0.25000) -- (-0.50000, -0.25000, 0.25000);
\draw[edge] (0.00000, 0.50000, -0.25000) -- (0.00000, 0.50000, 0.25000);
\draw[edge] (0.00000, 0.50000, -0.25000) -- (0.50000, 0.25000, -0.25000);
\draw[edge] (0.50000, -0.25000, -0.25000) -- (0.50000, 0.25000, -0.25000);
\draw[edge] (0.50000, -0.25000, -0.25000) -- (0.00000, -0.50000, -0.25000);
\draw[edge] (0.00000, -0.50000, -0.25000) -- (0.00000, -0.50000, 0.25000);
\draw[edge] (0.00000, -0.50000, -0.25000) -- (-0.50000, -0.25000, -0.25000);
\draw[edge] (0.00000, -0.50000, 0.25000) -- (-0.50000, -0.25000, 0.25000);
\draw[edge] (-0.50000, -0.25000, -0.25000) -- (-0.50000, -0.25000, 0.25000);
\node[vertex] at (-0.50000, 0.25000, -0.25000)     {};
\node[vertex] at (-0.50000, 0.25000, 0.25000)     {};
\node[vertex] at (0.00000, 0.50000, -0.25000)     {};
\node[vertex] at (0.00000, 0.50000, 0.25000)     {};
\node[vertex] at (0.50000, -0.25000, -0.25000)     {};
\node[vertex] at (0.50000, 0.25000, -0.25000)     {};
\node[vertex] at (0.00000, -0.50000, -0.25000)     {};
\node[vertex] at (0.00000, -0.50000, 0.25000)     {};
\node[vertex] at (-0.50000, -0.25000, -0.25000)     {};
\node[vertex] at (-0.50000, -0.25000, 0.25000)     {};

\draw (-0.5,0,0) node {$F_1,(1,0)$};
\draw (-0.3,0.35,-0.05) node {$F_2$};
\draw (0.6,0.4,0.05) node {$F_3,(-1,-a)$};
\draw (0.5,0,0) node {$F_4$};
\draw (0.5,-0.32,0.1) node {$F_5,(-1,-b)$};
\draw (-0.7,-0.5,-0.2) node {$F_6$};
\draw (0,0,0.25) node {$F_8,(0,-1)$};
\draw (0,0,-0.25) node {$F_7,(0,1)$};

\end{tikzpicture} 
\end{figure}

Now let $Q=H^2\times\Delta^{n-1}\subset\R^{n+1}$, where $H^2\subset\R^2$ is the hexagon and\\
$\Delta^{n-1}\subset\R^{n-1}$ is the $(n-1)$-dimensional simplex. Denote the edges of $H^2$ by $H_1,\dots,H_6$ in the counter-clockwise order, and denote the facets of $\Delta^{n-1}$ by\\
$\Delta_{1},\dots,\Delta_{n}$. Then the facets $F_1,\dots,F_{n+6}$ of $Q$ are
\[
H_1\times\Delta^{n-1},\dots,H_6\times\Delta^{n-1},H^2\times\Delta_{1},\dots,H^2\times\Delta_{n},
\]
respectively. Choose the distinguished facets $F_2,F_4,F_6$ of $Q$. Clearly,\\
$\vertex Q=\bigcup_{i=1}^{3}\vertex F_{n}$. Consider the function $\lambda:\mc{F}(Q)\setminus\lb F_1,F_3,F_5\rb\to\Z^{n}$ taking values equal the columns of the following $(n\times(n+3))$-matrix
\[
\begin{pmatrix}
1   & -1 & -1 &  0 &  0 &\dots&  0 &  0 \\
0   &  -a &  -b &  1 &  0 &\dots&  0 & -1 \\
0   &  0 &  0 &  0 &  1 &\dots&  0 & -1 \\
\vdots&\vdots&\vdots&\vdots&\vdots&\ddots&\vdots&\vdots\\
0   &  0 &  0 &  0 &  0 &\dots&  1 & -1 \\
\end{pmatrix},
\]
in the increasing order of facet indices.

\begin{pr}
$\lambda$ is an isotropy function on $\lb Q\setminus F_2,F_4,F_6 \rb$ satisfying the condition of Theorem \ref{thm:sarkar}.
\end{pr}
\begin{proof}
Follows immediately from Propostion \ref{pr:quas} and Lemma \ref{lm:check}.
\end{proof}

Now the construction of Theorem \ref{thm:sarkar} gives the following

\begin{thm}\label{thm:orbb}
$W^{2n+1}=W(\lb Q\setminus F_2,F_4,F_6\rb,\lambda)$ is a stably complex\\
$(2n+1)$-dimensional manifold with stably complex manifolds $P^{n}(a)$, $\overline{P^{n}(b)}$,\\
$(P^{n}(b-a),c_{\mc{T}}(b-a))$ being the the boundary components of $W$. 
\end{thm}


\end{document}